\documentclass[graybox]{svmult}

\usepackage{mathptmx} 
\usepackage{helvet} 
\usepackage{courier} 
\usepackage{type1cm} 
\usepackage{multicol} 
\usepackage[bottom]{footmisc}

\usepackage{amssymb}
\usepackage{mathrsfs}

\newcommand{\B}{\mathbf{B}}

\newcommand{\CC}{\mathbb{C}}
\newcommand{\ZZ}{\mathbb{Z}}
\newcommand{\PP}{\mathbb{P}}
\newcommand{\QQ}{\mathbb{Q}}

\newcommand{\Sym}{\mathfrak{S}}
\newcommand{\EEE}{\mathscr{E}}
\newcommand{\operatorname}[1]{\mathrm{#1}}
\newcommand{\Sing}{\operatorname{Sing}}

\newcommand{\Bir}{\operatorname{Bir}}
\newcommand{\Pic}{\operatorname{Pic}}

\newcommand{\Aut}{\operatorname{Aut}}
\newcommand{\Bs}{\operatorname{Bs}}
\newcommand{\Cr}{\operatorname{Cr}}

\newcommand{\rr}{\operatorname{r}}
\newcommand{\Fix}{\operatorname{Fix}}
\newcommand{\g}{\operatorname{g}}
\newcommand{\qq}{\mathbin{\sim_{\scriptscriptstyle{\QQ}}}}
\newcommand{\PGL}{PGL}
\newcommand{\GL}{GL}

\newcommand{\comp}{\mathbin{\scriptstyle{\circ}}}
\renewcommand\labelenumi{\textup{(\roman{enumi})}}

\newtheorem{ccorollary}[theorem]{Corollary}
\newtheorem{llemma}[theorem]{Lemma}
\newtheorem{pproposition}[theorem]{Proposition}

\begin{document}
\title*{$2$-elementary subgroups of the space Cremona group}

\author{Yuri Prokhorov \thanks{
 I acknowledge partial supports by
RFBR grants No. 11-01-00336-a,
 the grant of
 Leading Scientific Schools No. 4713.2010.1, 
 Simons-IUM fellowship,
 and 
 AG Laboratory SU-HSE, RF government 
 grant ag. 11.G34.31.0023.
 }}

\institute{
Steklov Mathematical Institute, 
8 Gubkina str., Moscow 119991
\and
Laboratory of Algebraic Geometry, SU-HSE, 
7 Vavilova str., Moscow 117312 \email{prokhoro@gmail.com}
}

\maketitle
\abstract{
We give a sharp bound for orders of elementary abelian 2-groups
of birational automorphisms of rationally connected threefolds.
}

\section{Introduction}\label{section-Introduction}
Throughout this paper we work over $\Bbbk$, an algebraically closed field of
characteristic $0$.
Recall that the \textit{Cremona group} $\Cr_n(\Bbbk)$ is the group of birational
transformations of the projective space $\PP^n_\Bbbk$. 
We are interested in finite subgroups of $\Cr_n(\Bbbk)$.
For $n=2$ these subgroups
are classified basically
(see \cite{Dolgachev-Iskovskikh} and references therein)
but for $n\ge 3$ 
the situation
becomes much more complicated.
There are only a few, very specific classification results (see e.g.
\cite{Prokhorov2009e}, \cite{Prokhorov2011a}, 
\cite{Prokhorov-2013-im}).

Let $p$ be a prime number.
A group $G$ is said to be \textit{$p$-elementary abelian} of rank $r$ 
if $G\simeq (\ZZ/p\ZZ)^r$.
In this case we denote $\rr(G):=r$.
A.~Beauville \cite{Beauville2007} obtained a sharp bound 
for the rank of $p$-elementary abelian subgroups of $\Cr_2(\Bbbk)$.
\begin{theorem}[{\cite{Beauville2007}}]
\label{theorem-cr2}
Let $G\subset \Cr_2(\Bbbk)$ be a $2$-elementary abelian subgroup.
Then $\rr(G) \le 4$.
Moreover, 
this bound is sharp and such groups $G$ with $\rr(G) = 4$ are classified up
to conjugacy in $\Cr_2(\Bbbk)$.
\end{theorem}
The author \cite{Prokhorov2011a} was able to get a similar bound for $p$-elementary abelian subgroups of $\Cr_3(\Bbbk)$
which is sharp for $p\ge 17$.

In this paper we improve this result in the case $p=2$.
We study $2$-elementary abelian subgroups
acting on rationally connected threefolds. In particular,
we obtain a \emph{sharp} bound for the rank of such subgroups in 
$\Cr_3(\Bbbk)$.
Our main result is the following.

\begin{theorem}\label{main}
Let $Y$ be a rationally connected variety over $\Bbbk$ and 
let $G\subset\Bir_\Bbbk(Y)$ be a $2$-elementary abelian group.
Then $\rr(G)\le 6$.
\end{theorem}

\begin{ccorollary}\label{main-Cremona}
Let $G\subset\Cr_3(\Bbbk)$ be a $2$-elementary abelian group.
Then $\rr(G)\le 6$ and the bound is sharp \textup(see Example \textup{\ref{example-Cr}}\textup).
\end{ccorollary}

Unfortunately we are not able to classify all the birational 
actions $G\hookrightarrow\Bir_\Bbbk(Y)$ as above 
attaining the bound $\rr(G)\le 6$ (cf. \cite{Beauville2007}). However, in some cases we
get a description of these ``extremal'' actions.

The structure of the paper is as follows.
In Sect. \ref{section-MMP} we reduce the problem to 
the study of biregular actions of $2$-elementary 
abelian groups on Fano-Mori fiber spaces and investigate 
the case of non-trivial base.
A few facts about actions of $2$-elementary 
abelian groups on Fano threefolds are discussed in
Sect. \ref{section-Fano-threefolds}.
In Sect. \ref{section-Non-Gorenstein-Fano} (resp. 
\S \ref{section-Gorenstein-Fano}) we study actions on
non-Gorenstein (resp. Gorenstein) Fano threefolds. 
Our main theorem is a direct consequence of 
Propositions \ref{propositionFano-Mori-fibration},
 \ref{proposition-non-Gorenstein}, and \ref{proposition-Gorenstein}.
 
\par\medskip\noindent
{\bf Acknowledgments.}
The work was completed during the author's stay at the
International Centre for Theoretical Physics, Trieste.
The author would like to thank ICTP for hospitality and support.

\section{Preliminaries}\label{section-Preliminaries}
\setcounter{theorem}{0}
\noindent{\bf Notation.}
\begin{itemize}
 \item 
For a group $G$, $\rr(G)$ denotes the minimal number of generators.
In particular, if $G$ is an elementary abelian $p$-group, then 
$G\simeq (\ZZ/p\ZZ)^{r(G)}$.
 \item 
$\Fix(G,X)$ (or simply $\Fix(G)$ if no confusion is likely) denotes the fixed point locus of an action 
of $G$ on $X$.
\end{itemize}

\noindent{\bf Terminal singularities.}
Recall that the \textit{index} of a terminal singularity $(X\ni P)$ is a minimal 
positive integer $r$ such that $K_X$ is a Cartier divisor at $P$.

\begin{llemma}
\label{lemma-fixed-point}
Let $(X\ni P)$ be a germ of a threefold terminal singularity 
and let $G\subset \Aut(X\ni P)$ be a $2$-elementary abelian subgroup. 
Then $\rr(G)\le 4$.
Moreover, if $\rr(G)= 4$, then $(X\ni P)$ is not a cyclic quotient singularity.
\end{llemma}
\begin{proof}
Let $ m$ be the index of $(X\ni P)$.
Consider the index-one cover $\pi\colon (X^\sharp\ni P^\sharp)\to 
(X\ni P)$ (see \cite{Reid-YPG1987}). 
Here $(X^\sharp\ni P^\sharp)$ is a terminal 
point of index $1$ (or smooth) and $\pi$ is a cyclic cover of degree $ m$ which is \'etale outside of $P$.
Thus $X\ni P$ is the quotient of $X^\sharp\ni P$ by a cyclic group of order $m$. 
If $m=1$, we take $\pi$ to be the identity map.
We may assume that $\Bbbk=\CC$ and then the map $X^\sharp\setminus \{P^\sharp\}\to 
X\setminus \{P\}$ can be regarded as the topological universal cover. 
Hence there exists a natural lifting
$G^\sharp\subset \Aut (X^\sharp\ni P^\sharp)$ fitting to the following exact sequence
$$
1\longrightarrow C_m \longrightarrow G^\sharp \longrightarrow G \longrightarrow 1,
\eqno(*) 
$$
where $C_m\simeq \ZZ/m\ZZ$.
We claim that $G^\sharp$ is abelian.
Assume the converse. Then $m\ge 2$.
The group $G^\sharp$ permutes 
eigenspaces of $C_m$. 
Let $T_{P^\sharp, X^\sharp}$ be the tangent space and let $n:=\dim T_{P^\sharp, X^\sharp}$
be the embedded dimension.
By the classification of three-dimensional terminal singularities \cite{Mori-1985-cla}, \cite{Reid-YPG1987}
we have one of the following:
$$
\begin{array}{l@{\quad}l}
(1)&\frac1m (a,-a,b),\quad n=3,\quad \gcd (a,m)=\gcd(b,m)=1; 
\\[3pt]
(2)&\frac1m (a,-a,b,0),\quad n=4,\quad \gcd (a,m)=\gcd(b,m)=1; 
\\[3pt]
(3)&\frac14 (a,-a,b,2),\quad n=4,\quad \gcd (a,2)=\gcd(b,2)=1,\quad m=4,
\end{array}\eqno(**)
$$
where $\frac1m (a_1,\dots, a_n)$ denotes the diagonal action 
\[
x_k \longmapsto \exp(2\pi i a_k/m)\cdot x_k,\quad k=1,\dots,n.
\]
Put $T= T_{P^\sharp, X^\sharp}$ in the first case and denote by 
$T\subset T_{P^\sharp, X^\sharp}$ the three-dimensional subspace $x_4=0$ in the 
second and the third cases. 
Then $C_m$ acts on $T$ freely outside of the origin and $T$ 
is $G^\sharp$-invariant.
By (*) 
we see that the derived subgroup
$[G^\sharp, G^\sharp]$ is contained in $C_m$.
In particular, $[G^\sharp, G^\sharp]$ 
is abelian and also acts on $T$ freely outside of the origin.
Assume that $[G^\sharp, G^\sharp]\neq \{1\}$.
Since $\dim T=3$, this implies that the representation of 
$G^\sharp$ on $T$ is irreducible
(otherwise $T$ has a one-dimensional invariant subspace, say $T_1$,
and the kernel of the map $G^\sharp\to \GL(T_1)\simeq \Bbbk^*$
must contain $[G^\sharp, G^\sharp]$). 
In particular, the eigenspaces of $C_m$ on $T$ have the same dimension.
Since $T$ is irreducible, the order of $G^\sharp$ is divisible by 
$3=\dim T$ and so $m>2$. In this case, by the above description of the action of $C_m$ on 
$T_{P^\sharp, X^\sharp}$ we get that there are exactly three
distinct eigenspaces $T_i\subset T$.
The action of $G^\sharp$ on the set $\{T_i\}$ induces a transitive homomorphism
$G^\sharp\to \Sym_3$ whose kernel contains $C_m$. Hence we have 
a transitive homomorphism $G\to \Sym_3$.
Since $G$ is a $2$-group, this is impossible.

Thus $G^\sharp$ is abelian. Then
\[
\rr(G)\le \rr(G^\sharp) \le \dim T_{P^\sharp, X^\sharp}.
\]
This proves our statement.
\hfill$\Box$ \end{proof}
\refstepcounter{theorem}
\begin{remark}\label{remark-fixed-point}
 If in the above notation the action of $G$ on $X$ is free in codimension one,
then $\rr(G)\le \dim T_{P^\sharp, X^\sharp}-1$.
\end{remark}

For convenience of references, we formulate the following easy result.
\begin{llemma}
\label{lemma-fixed-points}
Let $G$ be a $2$-elementary abelian group and 
let $X$ be a $G$-threefold with isolated singularities.
\begin{enumerate}
 \item 
If $\dim \Fix(G)>0$, then $\dim \Fix(G)+\rr(G)\le 3$.
\item 
Let $\delta\in G\setminus \{1\}$ and let $S\subset \Fix(\delta)$ be the union of two-dimensional components.
Then $S$ is $G$-invariant and smooth in codimension $1$.
\end{enumerate}
\end{llemma}\def\proofname{Sketch of the proof}
\begin{proof}
Consider the action of $G$ on 
the tangent space to $X$ at a general point of a component of $\Fix(G)$
(resp. at a general point of $\Sing(S)$). 
\hfill$\Box$ \end{proof}
\def\proofname{Proof}

\section{$G$-equivariant minimal model program.}\label{section-MMP}\setcounter{theorem}{0}
\refstepcounter{theorem}
\begin{definition}
Let $G$ be a finite group.
A \emph{$G$-variety} is a variety $X$ provided with 
a biregular faithful action of $G$.
We say that a normal $G$-variety 
$X$ is \emph{$G\QQ$-factorial} if any $G$-invariant Weil 
divisor on $X$ is $\QQ$-Cartier. 
\end{definition}

The following construction is standard (see e.g. \cite{Prokhorov2009e}).

Let $Y$ be a rationally connected three-dimensional algebraic variety
and let $G\subset \Bir(Y)$ be a finite subgroup.
Taking an equivariant compactification and running an equivariant 
minimal model program we get a $G$-variety $X$ and a $G$-equivariant 
birational map $Y \dashrightarrow X$, where $X$ 
has a structure a \textit{$G$-Fano-Mori fibration} $f: X\to B$.
This means that $X$ has at worst terminal $G\QQ$-factorial singularities, 
$f$ is a $G$-equivariant morphism with connected fibers, $B$ is normal,
$\dim B<\dim X$, 
 the anticanonical Weil
divisor $-K_{X}$ is ample over $B$, and the relative $G$-invariant Picard 
number $\rho(X)^G$ equals to one.
Obviously, in the case $\dim X=3$ we have the following possibilities:
\begin{enumerate}
\item[(C)]
$B$ is a rational surface and a general fiber $f^{-1}(b)$ is a conic;
\item[(D)]
$B\simeq \PP^1$ and a general fiber $f^{-1}(b)$ is a smooth del Pezzo surface;
\item[(F)]
$B$ is a point and $X$ is a \textit{$G\QQ$-Fano threefold}, that is,
$X$ is a Fano threefold with at worst terminal $G\QQ$-factorial singularities
and such that $\Pic(X)^G\simeq \ZZ$.
In this situation we say that $X$ is \textit{$G$-Fano threefold} if 
$X$ is Gorenstein, that is, $K_{X}$ is a Cartier divisor.
\end{enumerate}

\begin{pproposition}\label{propositionFano-Mori-fibration}
Let $G$ be a $2$-elementary abelian group and let $f: X\to B$ 
be a $G$-Fano-Mori fibration with $\dim X=3$ and $\dim B>0$.
Then $\rr( G)\le 6$.
Moreover, if $\rr( G)= 6$ and $Z\simeq \PP^1$, then 
a general fiber $f^{-1}(b)$ is a del Pezzo surface of degree $4$ or $8$.
\end{pproposition}

\begin{proof}
Let $G_f\subset G$ (resp. $G_B\subset \Aut(B)$) 
be the kernel (resp. the image) of the homomorphism 
$G\to \Aut(B)$. 
Thus $G_B$ acts faithfully on $B$ 
and $G_f$ acts faithfully on the 
generic fiber $F\subset X$ of $f$. 
Clearly, $G_f$ and $G_B$ are $2$-elementary groups with 
$\rr( G_f) +\rr( G_B)=\rr( G)$.
Assume that $B\simeq \PP^1$. Then $\rr( G_B)\le 2$ by the classification of finite subgroups of $PGL_2(\Bbbk)$.
By Theorem \ref{theorem-cr2} we have
$\rr( G_f)\le 4$. If furthermore $\rr( G)= 6$, then 
$\rr( G_f)= 4$ and the assertion about $F$ follows by Lemma \ref{lemma-Beauville-del-Pezzo} below. 
This proves 
our assertions in the case $B\simeq \PP^1$.
The case $\dim B=2$ is treated similarly.
\hfill$\Box$ \end{proof}

\begin{llemma}[cf. \cite{Beauville2007}]\label{lemma-Beauville-del-Pezzo}
Let $F$ be a del Pezzo surface 
and let $G\subset \Aut(F)$ be a $2$-elementary abelian group
with $\rr(F)\ge 4$.
Then $\rr(F)= 4$ and one of the following holds:
\begin{enumerate}
\item
$K_F^2=4$, $\rho(F)^G=1$;
\item
$K_F^2=8$, $\rho(F)^G=2$.
\end{enumerate}
\end{llemma}
\begin{proof}
Similar to \cite[\S 3]{Beauville2007}.
\hfill$\Box$ \end{proof}

\refstepcounter{theorem}
\begin{example}\label{example-Cr}
Let $F\subset \PP^4$ be the quartic del Pezzo surface given by 
$\sum x_i^2=\sum \lambda_i x_i^2=0$ with $\lambda_i\neq \lambda_j$ for $i\neq j$ and let $G_f\subset \Aut(F)$ 
be the $2$-elementary abelian subgroup generated by involutions $x_i\mapsto -x_i$.
Consider also a $2$-elementary abelian subgroup $G_B\subset \Aut(\PP^1)$
induced by a faithful representation $\mathrm Q_8\to GL_2(\Bbbk)$ of the quaternion group $\mathrm Q_8$. 
 Then $\rr(G_f)=4$, $\rr(G_B)=2$, and $G:=G_f\times G_B$
naturally acts on $X:=F\times \PP^1$. Two projections give us two 
structures of $G$-Fano-Mori fibrations of types (D) and (C).
This shows that the bound $\rr(G)\le 6$ in Proposition \ref{propositionFano-Mori-fibration}
is sharp. Moreover, $X$ is rational and so we have an embedding 
$G\subset \Cr_3(\Bbbk)$.
\end{example}

\section{Actions on Fano threefolds}\label{section-Fano-threefolds}\setcounter{theorem}{0}
\noindent
{\bf Main assumption.}
From now on we assume that we are in the case (F), that is,
$X$ is a $G\QQ$-Fano threefold. 

\refstepcounter{theorem}
\begin{remark}\label{definition-S}
The group $G$ acts naturally on $H^0(X,{-}K_X)$.
If $H^0(X,{-}K_X)\neq 0$, then there exists 
a $G$-semi-invariant section $0\neq s\in H^0(X,{-}K_X)$
(because $G$ is an abelian group).
This section defines an invariant member
$S\in |{-}K_X|$.
\end{remark}

\begin{llemma}
\label{lemma-hyp-sect-0}
Let $X$ be a $G\QQ$-Fano threefold, where $G$ is a $2$-elementary abelian
group with $\rr(G)\ge 5$. 
Let $S$ be an invariant effective Weil divisor 
such that $-(K_X+S)$ is nef. 
Then the pair $(X,S)$ is log canonical \textup(lc\textup).
In particular, $S$ is reduced.
If $-(K_X+S)$ is ample, then 
the pair $(X,S)$ is purely log terminal \textup(plt\textup).
\end{llemma}
\begin{proof}
Assume that the pair $(X,S)$ is not lc. 
Since $S$ is $G$-invariant and $\rho(X)^G=1$, we see that $S$ is 
numerically proportional to $K_X$. Hence $S$ is ample.
We apply quite standard
connectedness arguments of Shokurov \cite{Shokurov-1992-e-ba}
(see, e.g., \cite[Prop. 2.6]{Mori-Prokhorov-2008d}): 
for a suitable $G$-invariant boundary $D$, the pair 
$(X,D)$ is lc, the divisor $-(K_X+D)$ is ample,
and the minimal locus $V$ of log canonical singularities is also $G$-invariant. 
Moreover, $V$ is either a point or a smooth rational curve.
By Lemma \ref{lemma-fixed-point} we may assume that $G$ has no fixed points.
Hence, $V\simeq \PP^1$ and we have a map 
$\varsigma: G\to \Aut(\PP^1)$. 
By Lemma \ref{lemma-fixed-points} \ $\rr(\ker \varsigma)\le 2$.
Therefore, $\rr(\varsigma(G))\ge 3$.
This contradicts the classification of finite subgroups of $PGL_2(\Bbbk)$.

If $-(K_X+S)$ is ample and $(X,S)$ has a log canonical 
center of dimension $\le 1$, then by considering $(X,S'=S+\epsilon B)$, 
where $B$ is a suitable invariant divisor
and $0<\epsilon\ll 1$,
we get a non-lc pair $(X,S')$.
This contradicts the above considered case.
\hfill$\Box$ \end{proof}

\begin{ccorollary}
\label{corollary-hyp-sect-0}
Let $X$ be a $G\QQ$-Fano threefold, where $G$ is a $2$-elementary abelian
group with $\rr(G)\ge 6$ and 
let $S$ be an invariant Weil divisor. Then 
$-(K_X+S)$ is not ample. 
\end{ccorollary}
\begin{proof}
If $-(K_X+S)$ is ample, then by Lemma 
\ref{lemma-hyp-sect-0} the pair $(X,S)$ is plt.
By the adjunction principle \cite{Shokurov-1992-e-ba} the surface
$S$ is irreducible, normal and has only quotient singularities.
Moreover, $-K_S$ is ample.
Hence $S$ is
rational. We get a contradiction by Theorem \ref{theorem-cr2}
and Lemma \ref{lemma-fixed-points}(i).
\hfill$\Box$ \end{proof}

\begin{llemma}\label{lemma-K3}
Let $S$ be a K3 surface 
with at worst Du Val singularities and let $\Gamma\subset \Aut(S)$ be a $2$-elementary abelian
group. Then $\rr(\Gamma)\le 5$.
\end{llemma}

\begin{proof}
Let $\tilde S\to S$ be the minimal resolution.
Here $\tilde S$ is a smooth K3 surface and the action of $\Gamma$
lists to $\tilde S$.
Let $\Gamma_{\mathrm {s}}\subset \Gamma$ be the largest subgroup that 
acts trivially on $H^{2,0}(\tilde S)\simeq \CC$.
The group $\Gamma/\Gamma_{\mathrm {s}}$ is cyclic. Hence, $\rr(\Gamma/\Gamma_{\mathrm {s}})\le 1$. 
According to \cite[Th. 4.5]{Nikulin-1980-aut} we have 
$\rr(\Gamma_{\mathrm {s}})\le 4$. Thus $\rr(\Gamma)\le 5$.
\hfill$\Box$ \end{proof}

\begin{ccorollary}\label{corollary-lemma-K3}
Let $X$ be a $G\QQ$-Fano threefold, where $G$ is a $2$-elementary abelian group. 
Let $S\in |{-}K_X|$ be a $G$-invariant member.
If $\rr(G)\ge 7$, then the singularities of $S$ are worse than Du Val.
\end{ccorollary}

\begin{pproposition}\label{lemma-hyp-sect}
Let $X$ be a $G\QQ$-Fano threefold, where $G$ is a $2$-elementary abelian
group with $\rr(G)\ge 6$.
Let $S\in |{-}K_X|$ be a $G$-invariant member
and let $G_{\bullet}\subset G$ be the largest subgroup that 
acts trivially on the set of components of $S$. 
One of the following holds:
\begin{enumerate}
\item 
$S$ is a K3 surface with at worst Du Val singularities, then 
$S\subset \Fix(\delta)$ for some $\delta\in G\setminus
\{1\}$ and $G/\langle \delta\rangle$ faithfully acts on $S$.
In this case $\rr(G)=6$.
\item 
The surface $S$ is reducible \textup(and reduced\textup).
The group $G$ acts transitively on the components of $S$
and $G_{\bullet}$ acts faithfully on each component $S_i\subset S$.
There are two subcases:
\begin{enumerate}
\item 
any component $S_i\subset S$ is rational and $\rr(G_{\bullet})\le 4$.
\item 
any component $S_i\subset S$ is
birationally ruled over an elliptic curve and $\rr(G_{\bullet})\le 5$.
\end{enumerate}
\end{enumerate}
\end{pproposition}

\begin{proof}
By Lemma \ref{lemma-hyp-sect-0} the pair $(X,S)$ is lc.
Assume that $S$ is normal (and irreducible).
By the adjunction formula $K_S\sim 0$.
We claim that $S$ has at worst Du Val singularities. Indeed, otherwise 
by the Connectedness Principle \cite[Th. 6.9]{Shokurov-1992-e-ba}
$S$ has at most two non-Du Val points. 
These points are fixed by an index two subgroup 
$G'\subset G$. This contradicts Lemma \ref{lemma-fixed-point}.
Taking Lemma \ref{lemma-K3} into account we get the case (i).

Now assume that $S$ is not normal.
Let $S_i\subset S$ be an irreducible component (the case $S_i=S$
is not excluded).
If the action on components $S_i\subset S$ 
is not transitive, 
there is an invariant divisor $S'<S$.
Since $X$ is $G\QQ$-factorial and $\rho(X)^G=1$,
the divisor $-(K_X+S')$ is ample.
This contradicts Corollary \ref{corollary-hyp-sect-0}.
By Lemma \ref{lemma-fixed-points}(ii) the action of $G_{\bullet}$ 
on each component $S_i$ is faithful.

If $S_i$ is a rational surface, then $\rr(G_{\bullet})\le 4$ by 
Theorem \ref{theorem-cr2}.
Assume that $S_i$ is not rational.
Let $\nu\colon S'\to S_i$
be the normalization. 
Write $0\sim \nu^*(K_X+S)=K_{S'}+D'$,
where $D'$ is the \emph{different}, see \cite[\S 3]{Shokurov-1992-e-ba}. 
Here $D'$ is an effective reduced divisor
and the pair is lc \cite[3.2]{Shokurov-1992-e-ba}.
Since $S$ is not normal, $D'\neq 0$.
Consider the minimal resolution 
$\mu\colon \tilde S\to S'$ and let $\tilde D$ be the crepant 
pull-back of $D'$, that is,
$\mu_*\tilde D=D'$ and
\[
K_{\tilde S}+\tilde D=\mu^*(K_{S'}+D')\sim 0.
\]
Here $\tilde D$ is again an effective reduced divisor.
Hence $\tilde S$ is a ruled surface. Consider the Albanese map $\alpha: \tilde S\to C$. 
Let $\tilde D_1\subset \tilde D$ be an $\alpha$-horizontal
component. By the adjunction formula $\tilde D_1$ is an elliptic curve and so $C$ is. 
Let $\Gamma$ be the image of $G_{\bullet}$ in $\Aut(C)$.
Then $\rr(\Gamma)\le 3$ and so $\rr(G_{\bullet})\le 5$.
So, the last assertion is proved.
\hfill$\Box$ \end{proof}

\section{Non-Gorenstein Fano threefolds}\label{section-Non-Gorenstein-Fano}\setcounter{theorem}{0}
Let $G$ be a $2$-elementary abelian group and let $X$ be 
$G\QQ$-Fano threefold.
In this section we consider the case where $X$ is non-Gorenstein, i.e., it has at least one 
terminal point of index $>1$.
We denote by 
$\Sing'(X)=\{P_1,\dots,P_M\}$ the set of non-Gorenstein points and by
$\B=\B(X)$ the 
\textit{basket of singularities} \cite{Reid-YPG1987}.
By $\B(X,P_i)$ we denote the basket of singularities at a point $P_i\in X$.

\begin{pproposition}
\label{proposition-non-Gorenstein}
Let $X$ be a non-Gorenstein Fano threefold with terminal singularities.
Assume that $X$ admits a faithful action of 
a $2$-elementary abelian group $G$ with $\rr(G)\ge 6$. Then 
$\rr(G)=6$, $G$ transitively acts on $\Sing'(X)$, $|{-}K_X |\neq\emptyset$,
and the configuration of non-Gorenstein singularities is described
below.
\begin{enumerate}\renewcommand\labelenumi{\textup{(\arabic{enumi})}}
\renewcommand\theenumi{\textup{(\arabic{enumi})}}
\item\label{proposition-non-Gorenstein-1}
$M=8$, $\B(X,P_i)= \left\{\frac12(1,1,1)\right\}$;
\item\label{proposition-non-Gorenstein-2}
$M=8$, $\B(X,P_i)=\left\{ \frac13(1,1,2)\right\}$;
\item\label{proposition-non-Gorenstein-3}
$M=4$, $\B(X,P_i)=\left\{ 2\times \frac12(1,1,1)\right\}$;
\item\label{proposition-non-Gorenstein-4}
$M=4$, $\B(X,P_i)=\left\{ 2\times \frac13(1,1,2)\right\}$;
\item \label{proposition-non-Gorenstein-5}
$M=4$, $\B(X,P_i)=\left\{ 3\times \frac12(1,1,1)\right\}$;
\item \label{proposition-non-Gorenstein-6}
$M=4$, 
$\B(X,P_i)= \left\{\frac14(1,-1,1),\, \frac12(1,1,1)\right\}$.
\end{enumerate}
\end{pproposition}

\begin{proof}
Let $P^{(1)},\dots, P^{(n)}\in \Sing'(X)$ be representatives of distinct $G$-orbits and
let $G_i$ be the stabilizer of $P^{(i)}$.
Let $r:=\rr(G)$, $r_i:=\rr(G_i)$, and let $m_{i,1},\dots,m_{i,\nu_i}$ be the indices of points in the basket of $P^{(i)}$. 
We may assume that $m_{i,1}\ge \cdots \ge m_{i,\nu_i}$
By the orbifold Riemann-Roch formula \cite{Reid-YPG1987} and
a form of Bogomolov-Miyaoka 
inequality \cite{Kawamata-1992bF}, \cite{KMMT-2000} we have
$$
\sum_{i=1}^n 2^{r-r_i} \sum_{j=1}^{\nu_i} \left( m_{i,j}-\frac1{m_{i,j}} \right)<24.
\eqno(***)
$$
If $P$ is a cyclic quotient singularity, then $\nu_i=1$ and by Lemma \ref{lemma-fixed-point} $r_i\le 3$.
If $P$ is not a cyclic quotient singularity, then $\nu_i\ge 2$ and 
again by Lemma \ref{lemma-fixed-point} \ $r_i\le 4$. Since $m_{i,j}-1/m_{i,j}\ge 3/2$, in both cases we have 
\[
2^{r-r_i} \sum_{j=1}^{\nu_i} \left( m_{i,j}-\frac1{m_{i,j}} \right)\ge 3\cdot 2^{r-4}\ge 12.
\]
Therefore, $n=1$, i.e. $G$ transitively acts on $\Sing'(X)$, and $r=6$.

If $P$ is not a point of type $cAx/4$ (i.e. it is not as in (3) of 
(**),
then 
by the classification of terminal singularities \cite{Reid-YPG1987}
$m_{1,1}= \cdots = m_{1,\nu_i}$ and 
(***) 
has the form
\[
24> 2^{6-r_1} \nu_1 \left( m_{1,1}-\frac1{m_{1,1}} \right)\ge 
 8 \left( m_{1,1}-\frac1{m_{1,1}} \right).
\]
Hence $r_1\ge 3$, $\nu_1\le 3$, $m_{1,1}\le 3$, and
$3\cdot 2^{r_1-3}\ge \nu_1 m_{1,1}$.
If $r_1= 3$, then $\nu_1=1$.
If $r_1=4$, then $\nu_1\ge 2$ and $\nu_1m_{1,1}\le 6$.
This gives us the possibilities \ref{proposition-non-Gorenstein-1}-\ref{proposition-non-Gorenstein-5}.

Assume that $P$ is a point of type $cAx/4$.
Then $m_{1,1}=4$, $\nu_1>1$, and $m_{1,j}=2$ for $1<j\le \nu_1$.
Thus (***) 
has the form
\[
24> 2^{6-r_1} \left(\frac {15}4 + \frac32 (\nu_1-1) \right)=
 2^{4-r_1} \left(9 + 6\nu_1 \right).
\]
We get $\nu_1=2$, $r_1=4$, i.e. the case \ref{proposition-non-Gorenstein-6}.

Finally, the computation of $\dim |{-}K_X|$ follows by the orbifold Riemann-Roch formula \cite{Reid-YPG1987}
\[ 
\dim|{-}K_X | =
-\frac 12 K_X^3 + 2 - \sum_{P\in \B(X)} \frac{ b_P (m_P - b_P )}{2m_P}. 
\]
\end{proof}

\section{Gorenstein Fano threefolds}\label{section-Gorenstein-Fano}\setcounter{theorem}{0}
The main result of this section is the following:
\begin{pproposition}\label{proposition-Gorenstein}
Let $G$ be a $2$-elementary abelian group and let $X$ be a 
\textup(Gorenstein\textup) $G$-Fano threefold.
Then $\rr(G)\le 6$. Moreover, if $\rr(G)= 6$, then
$\Pic(X)=\ZZ\cdot K_X$ and ${-}K_X^3\ge 8$. 
\end{pproposition}

Let $X$ be a Fano threefold with at worst Gorenstein terminal singularities.
 Recall that the number 
\[
\iota(X):=\max \{ i\in \ZZ \mid {-}K_X\sim i A,\ A\in \Pic(X) \}
\]
 is called the \textit{Fano index} of $X$. 
The integer $g=\g(X)$ such that ${-}K_X^3=2g-2$ is called the \textit{genus} of $X$.
It is easy to see that $\dim |{-}K_X|=g+1$ \cite[Corollary 2.1.14]{Iskovskikh-Prokhorov-1999}.
In particular, $|{-}K_X|\neq\emptyset$.

\par\medskip
\noindent{\bf Notation.} 
Throughout this section $G$ denotes a $2$-elementary abelian group and 
$X$ denotes a Gorenstein $G$-Fano threefold.
There exists an invariant member $S\in |{-}K_X|$
(see \ref{definition-S}). 
We write $S=\sum_{i=1}^N S_i$, where the $S_i$ are irreducible components.
Let $G_{\bullet}\subset G$ be the kernel of the homomorphism $G\to \Sym_N$
induced by the action of $G$ on $\{S_1,\dots,S_N\}$.
Since $G$ is abelian and the action of $G$ on $\{S_1,\dots,S_N\}$ is transitive, 
the group $G_{\bullet}$ coincides with the stabilizer of any $S_i$.
Clearly, $N=2^{\rr(G)-\rr(G_{\bullet})}$.
If $\rr(G)\ge 6$, then by Proposition \ref{lemma-hyp-sect} we have 
$\rr(G_{\bullet})\le 5$ and so $N\ge 2^{\rr(G)-5}$.

\begin{llemma}\label{lifting-SL}
Let $G\subset \Aut(\PP^n)$ be a $2$-elementary subgroup and $n$ is even.
Then $G$ is conjugate to a diagonal subgroup.
In particular, $\rr(G)\le n$.
\end{llemma}
\begin{proof}
Let $G^\sharp\subset SL_{n+1}(\Bbbk)$ be the lifting of $G$ and let 
$G'\subset G^\sharp$ be a Sylow $2$-subgroup. Then $G'\simeq G$.
Since $G'$ is abelian, the representation $G'\hookrightarrow SL_{n+1}(\Bbbk)$
is diagonalizable. 
\hfill$\Box$ \end{proof}

\begin{ccorollary}\label{quadric}
Let $Q\subset \PP^4$ be a quadric and let $G\subset \Aut(Q)$ be a $2$-elementary subgroup.
Then $\rr(G)\le 4$.
\end{ccorollary}

\begin{llemma}\label{P3}
Let $G\subset \Aut(\PP^3)$ be a $2$-elementary subgroup.
Then $\rr(G)\le 4$.
\end{llemma}
\begin{proof}
Assume that $\rr(G)\ge 5$. Take any element $\delta \in G\setminus \{1\}$.
By Lemma \ref{lemma-fixed-point} the group $G$ has no fixed points.
Since the set $\Fix(\delta)$ is $G$-invariant, 
$\Fix(\delta)=L_1\cup L_2$, where $L_1,\, L_2\subset \PP^3$ 
are skew lines.

Let $G_1\subset G$ be the stabilizer of $L_1$. There is a subgroup 
$G_2\subset G_1$ of index $2$ having a fixed point $P\in L_1$.
Thus $\rr(G_2)\ge 3$ and the ``orthogonal'' 
plane $\Pi$ is $G_2$-invariant. By Lemma \ref{lifting-SL} there exists an element
$\delta'\in G_2$ that acts trivially on $\Pi$, i.e. $\Pi\subset \Fix(\delta')$. 
But then $\delta'$ has a fixed point, a contradiction.
\hfill$\Box$ \end{proof}

\begin{llemma}\label{lemma-Bs}
If $\Bs|{-}K_X|\neq \emptyset$, then $\rr(G)\le 4$. 
\end{llemma}
\begin{proof}
By \cite{Shin1989} the base locus $\Bs|{-}K_X|$ is either a single point 
or a rational curve. In both cases $\rr(G)\le 4$ by
Lemmas \ref{lemma-fixed-point} and \ref{lemma-fixed-points}.
\hfill$\Box$ \end{proof}

\begin{llemma}\label{lemma-base-point-free}
If ${-}K_X$ is not very ample, then $\rr(G)\le 5$. 
\end{llemma}
\begin{proof}
Assume that $\rr(G)\ge 6$. By Lemma \ref{lemma-Bs} 
the linear system $|{-}K_X|$ is base point free.
It is easy to show that $|{-}K_X|$ defines a double cover 
$\phi: X\to Y\subset \PP^{g+1}$ (cf. \cite[Ch. 1, Prop. 4.9]{Iskovskikh-1980-Anticanonical}).
Here $Y$ is a variety of degree $g-1$ in $\PP^{g+1}$,
a variety of minimal degree. Let $\bar G$ be the image of $G$ in $\Aut(Y)$.
Then $\rr(\bar G)\ge \rr(G)-1$.
If $g=2$ (resp. $g=3$), then $Y=\PP^3$ (resp. $Y\subset \PP^4$ is a quadric) 
and $\rr(G)\le 5$ by Lemma \ref{P3}
(resp. by Corollary \ref{quadric}).
Thus we may assume that $g\ge 4$. 
If $Y$ is smooth, then according to the Enriques theorem 
(see, e.g., \cite[Th. 3.11]{Iskovskikh-1980-Anticanonical})
$Y$ is a rational scroll
$\PP_{\PP^1}(\EEE)$, where 
$\EEE$ is a rank $3$ vector bundle on $\PP^1$.
Then $X$ has a $G$-equivariant projection to a curve.
This contradicts $\rho(X)^G=1$.
Hence $Y$ is singular. In this case, $Y$ is a projective cone (again by the Enriques theorem).
If its vertex $O\in Y$ is zero-dimensional, then $\dim T_{O,Y}\ge 5$.
On the other hand, $X$ has only hypersurface singularities.
Therefore the double cover $X\to Y$ is not \'etale over $O$ and so
$G$ has a fixed point 
on $X$. This contradicts Lemma \ref{lemma-fixed-point}.
Thus $Y$ is a cone over a curve 
with vertex along a line $L$. 
As above, $L$ must be contained in the branch divisor
and so $L':=\phi^{-1}(L)$ is a $G$-invariant rational curve.
Since the image of $G$ in $\Aut(L')$ is a $2$-elementary abelian group
of rank $\le 2$, by Lemma \ref{lemma-fixed-points} we have $\rr(G)\le 4$. 
\hfill$\Box$ \end{proof}

\refstepcounter{theorem}
\begin{remark}
\label{iota}
Recall that for a Fano threefold $X$ with at worst Gorenstein terminal singularities one has
$\iota(X)\le 4$. Moreover, $\iota(X)=4$ if and only if $X\simeq \PP^3$ and 
$\iota(X)=3$ if and only if $X$ is a quadric in $\PP^4$ \cite{Iskovskikh-Prokhorov-1999}.
In these cases
we have $\rr(G)\le 4$ by Lemma \ref{P3} and Corollary \ref{quadric},
respectively.
Assume that $\iota(X)=2$. Then $X$ is so-called 
\emph{del Pezzo threefold}.
Let $A:=-\frac 12 K_X$. The number $d:=A^3$ is called the \textit{degree} of $X$.
\end{remark}

\begin{llemma}\label{corollary-action}
Assume that the divisor ${-}K_X$ is very ample, $\rr(G)\ge 6$, and the action of $G$
on $X$ is not free in codimension $1$.
Let $\delta\in G$ be an element 
such that $\dim \Fix(\delta)=2$ and let $D\subset \Fix(\delta)$ be the union of all two-dimensional components. 
Then $\rr(G)= 6$ and $D$ is a Du Val member of $|{-}K_X|$.
Moreover, $\iota(X)=1$ except, possibly, for the case where $\iota(X)=2$
and ${-}\frac12 K_X$ is not very ample.
\end{llemma}

\begin{proof}
Since $G$ is abelian, $\Fix(\delta)$ and $D$ are $G$-invariant and so
${-}K_X\qq \lambda D$ for some $\lambda \in \QQ$. 
In particular, $D$ is a $\QQ$-Cartier divisor.
Since $X$ has only terminal Gorenstein singularities, $D$ must be Cartier.
Clearly, $D$ is smooth outside of $\Sing(X)$.
Further, 
 $D$ is ample and 
 so it must be connected.
 Since $D$ is a reduced Cohen-Macaulay scheme with $\dim \Sing (D)\le 0$, 
 it is irreducible and normal.

Let $X\hookrightarrow \PP^{g+1}$ the anticanonical embedding. 
The action of $\delta$ on $X$ is induced by an action of a linear involution of $\PP^{g+1}$.
There are two disjointed linear subspaces $V_+,\, V_-\subset \PP^{g+1}$
of $\delta$-fixed points and the divisor $D$ is contained in one of them.
This means that $D$ is a component of a hyperplane section $S\in |{-}K_X|$
and so $\lambda \ge 1$.
Since $\rr(G)\ge 6$, by Corollary
\ref{corollary-hyp-sect-0} we have $\lambda =1$ and ${-}K_X\sim D$
(because $\Pic(X)$ is a torsion free group).
Since $D$ is irreducible, the case (i) of Proposition \ref{lemma-hyp-sect}
holds.

Finally, if $\iota(X)>1$, then by Remark \ref{iota} we have 
$\iota(X)=2$.
If furthermore the divisor $A$ is very ample,
then it defines an embedding 
$X \hookrightarrow \PP^N$ so that $D$ spans $\PP^N$. In this case the action of $\delta$ must be trivial,
a contradiction.
\hfill$\Box$ \end{proof}

\begin{llemma}\label{lemma-rho234}
If $\rho(X)>1$, then $\rr(G)\le 5$.
\end{llemma}
\begin{proof}
We use the classification of $G$-Fano threefolds with $\rho(X)>1$
\cite{Prokhorov-GFano-2}.
By this classification $\rho(X)\le 4$.
Let $G_0$ be the kernel of the action of $G$ on $\Pic(X)$.

Consider the case $\rho(X)=2$. Then $[G:G_0]=2$.
In the cases (1.2.1) and (1.2.4) of \cite{Prokhorov-GFano-2} the variety $X$ 
has a structure of $G_0$-equivariant conic bundle over $\PP^2$.
As in Proposition \ref{propositionFano-Mori-fibration} we have $\rr(G_0)\le 4$ and $\rr(G)\le 5$ in these cases.
In the cases (1.2.2) and (1.2.3) of \cite{Prokhorov-GFano-2} the variety $X$ 
has two birational contractions to $\PP^3$ and a quadric $Q\subset \PP^4$,
respectively. As above we get $\rr(G)\le 5$ by Lemma
\ref{P3} and Corollary \ref{quadric}.

Consider the case $\rho(X)=3$. 
We show that in this case $\Pic(X)^G\not\simeq \ZZ$
(and so this case does not occur).
Since $G$ is a $2$-elementary abelian group, its action on 
$\Pic(X)\otimes \QQ$ is diagonalizable. 
Since, $\Pic(X)^G=\ZZ\cdot K_X$, the group
$G$ contains an element $\tau$
that acts on $\Pic(X)\simeq \ZZ^3$ as the reflection with respect to 
the orthogonal complement to $K_X$. Since the group $G$ preserves the natural 
bilinear form $\langle \mathbf x_1,\, \mathbf x_2\rangle := \mathbf x_1\cdot \mathbf x_2\cdot K_X$, 
the action must be as follows
\[
\tau: \mathbf x \longmapsto \mathbf x- \lambda K_X,
\qquad \lambda= \frac{2 \mathbf x\cdot K_X^2}{K_X^3}.
\]
Hence $\lambda K_X$ is an integral element for any $\mathbf x\in \Pic(X)$.
This gives a contradiction in all cases (1.2.5)-(1.2.7) of \cite[Th. 1.2]{Prokhorov-GFano-2}.
For example, in the case (1.2.5) of \cite[Th. 1.2]{Prokhorov-GFano-2}
our variety $X$ has a structure (non-minimal) del Pezzo fibration of degree $4$
and ${-}K_X^3=12$.
For the fiber $F$ we have $F\cdot K_X^2=K_F^2=4$
and $\lambda K_X$ is not integral, a contradiction.

Finally, consider the case $\rho(X)=4$. 
Then according to \cite{Prokhorov-GFano-2} $X$ is a divisor of multidegree $(1,1,1,1)$ in
$(\PP^1)^4$. 
All the projections $\varphi_i : X\to \PP^1$, $i=1,\dots,4$ 
are $G_0$-equivariant. We claim that natural maps ${\varphi_{i}}_* :G_0\to \Aut(\PP^1)$
are injective. Indeed, assume that ${\varphi_{1}}_*(\vartheta)$ is the identity map 
in $\Aut(\PP^1)$
for some $\vartheta\in G$. This means that
$\vartheta\comp \varphi_1=\varphi_1$.
Since $\Pic(X)^G=\ZZ$, the group $G$ permutes the classes 
$\varphi_i^*\mathscr O_{\PP^1}(1)\in \Pic(X)$. Hence,
for any $i=1,\dots,4$, there exists $\sigma_i\in G$ such that
$\varphi_i=\varphi_1\comp \sigma_i$.
Then 
\[
\vartheta\comp \varphi_i= \vartheta\comp \varphi_1\comp \sigma_i= \varphi_1\comp \sigma_i=\varphi_i.
\]
Hence, 
${\varphi_{i}}_*(\vartheta)$ is the identity for any $i$.
Since $\varphi_1\times \cdots\times \varphi_4$ is an embedding,
$\vartheta$ must be the identity as well. 
This proves our clam. 
Therefore, $\rr(G_0)\le 2$.
The group $G/G_0$ acts on $\Pic(X)$ faithfully.
By the same reason as above, an element of $G/G_0$
cannot act as the reflection with respect to 
 $K_X$. Therefore, $\rr(G/G_0)\le 2$ and $\rr(G)\le 4$. 
\hfill$\Box$ \end{proof}

\begin{llemma}\label{lemma-Gorenstein-index}
If $\iota(X)=2$,
then $\rr(G)\le 5$.
\end{llemma}
\begin{proof}
By Lemma \ref{lemma-rho234} we may assume that $\rho(X)=1$.
Let $d$ be the degree of $X$.
Since $\rho(X)=1$, we have $d\le 5$ (see e.g. \cite{Prokhorov-GFano-1}).
Consider the possibilities for $d$ case by case. We use the classification 
(see \cite{Shin1989} and \cite{Prokhorov-GFano-1}).

If $d=1$, then the linear system $|A|$ has a unique base point.
This point is smooth and must be $G$-invariant.
By Lemma \ref{lemma-fixed-point} $\rr(G)\le 3$.
If $d=2$, then 
the linear system $|A|$ defines a double cover 
$\varphi: X\to \PP^3$. Then the image of $G$ in $\Aut(\PP^3)$ is a $2$-elementary
group $\bar G$ with $\rr(\bar G)\ge \rr(G)-1$, where $\rr(\bar G)\le 4$ by Lemma \ref{P3}.
If $d=3$, then 
$X=X_3\subset \PP^4$ is a cubic hypersurface.
By Lemma \ref{lifting-SL} $\rr(G)\le 4$.
If $d=5$, then 
$X$ is smooth, unique up to isomorphism, and $\Aut(X)\simeq \PGL_2(\Bbbk)$
(see \cite{Iskovskikh-Prokhorov-1999}).

Finally, consider the case $d=4$.
Then $X=Q_1\cap Q_2\subset \PP^5$ is an intersection of two quadrics (see e.g. 
\cite{Shin1989}).
Let $\mathscr Q$ be the pencil generated by $Q_1$ and $Q_2$.
Since $X$ has a isolated singularities and it is not a cone,
a general member of $\mathscr Q$ is smooth by Bertini's theorem
and for any member $Q\in \mathscr Q$ we have $\dim \Sing(Q)\le 1$.
Let $D$ be the divisor of degree $6$ on $\mathscr Q\simeq \PP^1$
given by the vanishing of the determinant.
The elements of $\operatorname{Supp} (D)$ are exactly degenerate quadrics.
Clearly, for any point $P\in \Sing(X)$ there exists a unique quadric $Q\in \mathscr Q$
which is singular at $P$. This defines a map $\pi:\Sing(X)\to \operatorname{Supp} (D)$.
Let $Q\in \operatorname{Supp} (D)$.
Then $\pi^{-1}(Q)=\Sing(Q)\cap X=\Sing(Q)\cap Q'$, 
where $Q'\in \mathscr Q$, $Q'\neq Q$. In particular, $\pi^{-1}(Q)$ consists of 
at most two points. Hence the cardinality of $\Sing(X)$ is at most $12$.

Assume that $\rr(G)\ge 6$. 
Let $S\in |-K_X|$ be an invariant member.
We claim that $S\supset \Sing(X)$ and $\Sing(X)\neq \emptyset$. Indeed, 
otherwise $S\cap \Sing(X)=\emptyset$.
By Proposition \ref{lemma-hyp-sect} \ $S$ is reducible: $S=S_1+\cdots+S_N$,
$N\ge 2$.
Since $\iota(X)=2$, we get $N=2$ and $S_1\sim S_2$, i.e. $S_i$ is a hyperplane section of $X\subset \PP^5$.
As in the proof of Corollary \ref{corollary-hyp-sect-0} we see that $S_i$ is rational.
This contradicts Proposition \ref{lemma-hyp-sect} (ii). 
Thus $\emptyset\neq \Sing(X)\subset S$.
By Lemma \ref{corollary-action}
the action of $G$ on $X$ 
is free in codimension $1$.
By Remark \ref{remark-fixed-point}
for the stabilizer $G_P$ of a point $P\in \Sing(X)$ we have $\rr(G_P)\le 3$. 
Then by the above estimate the variety $X$ has exactly
$8$ singular points and $G$ acts on $\Sing(X)$ transitively.

Note that our choice of $S$ is not unique:
there is a basis $s^{(1)}$, \dots, $s^{(g+2)}\in H^0(X,{-}K_X)$
consisting of eigensections. This basis gives us $G$-invariant divisors 
$S^{(1)}$, \dots, $S^{(g+2)}$ generating $|{-}K_X|$.
By the above $\Sing(X)\subset S^{(i)}$ for all $i$.
Thus $\Sing(X)\subset \cap S^{(i)}=\Bs |{-}K_X|$.
This contradicts the fact that $-K_X$ is very ample.
\hfill$\Box$ \end{proof}

\refstepcounter{theorem}
\begin{example}\label{example-del-pezzo-3-folds}
The bound $\rr(G)\le 5$ in the above lemma is sharp.
Indeed, let $X\subset \PP^5$ be the variety given by 
$\sum x_i^2=\sum \lambda_i x_i^2=0$ with $\lambda_i\neq \lambda_j$ for $i\neq j$ and 
let $G\subset \Aut(X)$ 
be the $2$-elementary abelian subgroup generated by involutions $x_i\mapsto -x_i$.
Then $\rr(G)=5$.
\end{example}

From now on we assume that $\Pic(X)=\ZZ\cdot K_X$. Let $g:=\g(X)$.

\begin{llemma}
If $g\le 4$, then $\rr(G)\le 5$. If $g=5$, then $\rr(G)\le 6$. 
\end{llemma}

\begin{proof}
We may assume that ${-}K_X$ is very ample.
Automorphisms of $X$ are induced by projective transformations 
of $\PP^{g+1}$ that preserve $X\subset \PP^{g+1}$.
On the other hand, there is a natural representation
of $G$ on $H^0(X,{-}K_X)$ which is faithful.
Thus the composition 
\[
\Aut(X)\hookrightarrow \GL(H^0(X,{-}K_X))=\GL_{g+2}(\Bbbk)\to \PGL_{g+2}(\Bbbk)
\]
is injective. Since $G$ is abelian, its image $\bar G\subset \GL_{g+2}(\Bbbk)$ is contained in 
a maximal torus and by the above $\bar G$ contains no scalar matrices.
Hence, $\rr(G)\le g+1$.
\hfill$\Box$ \end{proof}

\refstepcounter{theorem}
\begin{example}
Let $G$ be the $2$-torsion subgroup of the diagonal torus of 
$\PGL_7 (\Bbbk)$. 
Then $X$ faithfully acts on the Fano threefold in $\PP^6$ 
given by the equations $\sum x_i^2=\sum \lambda_i x_i^2=\sum \mu_i x_i^2=0$.
This shows that the bound $\rr(G)\le 6$ in the above lemma
is sharp. Note however that $X$ is not rational
if it is smooth \textup{\cite{Beauville1977}}.
Hence in this case our construction does not give any embedding 
of $G$ to $\Cr_3(\Bbbk)$. 
\end{example}

\begin{llemma}\label{lemma-singularities}
If in the above assumptions 
$\g(X)\ge 6$, then $X$ has at most $29$ singular points.
\end{llemma}
\begin{proof}
According to \cite{Namikawa-1997} the variety $X$ has 
a \emph{smoothing}. This means that there exists a flat family $\mathfrak X\to \mathfrak T$ over 
a smooth one-dimensional base $\mathfrak T$
with special fiber $X=\mathfrak X_0$ and
smooth general fiber $X_t=\mathfrak X_t$.
Using the classification of Fano threefolds 
\cite{Iskovskikh-1980-Anticanonical} 
(see also \cite{Iskovskikh-Prokhorov-1999}) we obtain $h^{1,2} (X_t )\le 10$.
Then by \cite{Namikawa-1997} we have 
\[
\# \Sing (X)\le 21-\frac 12 \operatorname{Eu}(X_t)= 20 - \rho(X_t ) + h^{1,2} (X_t )\le 29.
\]
\end{proof}

\def\proofname{Proof of Proposition \textup{\ref{proposition-Gorenstein}}}
\begin{proof}
Assume that $\rr(G)\ge 7$.
Let $S\in |{-}K_X|$ be an invariant member.
By Corollary \ref{corollary-lemma-K3} the singularities of $S$ are worse than Du Val.
So $S$ satisfies the conditions (ii) of Proposition \ref{lemma-hyp-sect}.
Write $S=\sum _{i=1}^N S_i$.
By Proposition \ref{lemma-hyp-sect} the group $G_{\bullet}$ acts on $S_i$ faithfully and 
\[
N=2^{\rr(G)-\rr(G_{\bullet})}\ge 4.
\]
First we consider the case where $X$ is smooth near $S$.
Since $\rho(X)=1$,
the divisors $S_i$'s are linear equivalent to each other
and so $\iota (X)\ge 4$. This contradicts 
Lemma \ref{lemma-Gorenstein-index}.

Therefore, $S\cap \Sing(X)\neq \emptyset$.
By Lemma \ref{corollary-action} the action of $G$ on $X$
is free in codimension $1$
and by Remark \ref{remark-fixed-point}
we see that $\rr(G_P)\le 3$, where $G_P$ is the stabilizer of a point $P\in \Sing(X)$. 
Then by Lemma \ref{lemma-singularities} the variety $X$ has exactly
$16$ singular points and $G$ acts on $\Sing(X)$ transitively.
Since $S\cap \Sing(X)\neq \emptyset$, we have $\Sing(X)\subset S$.
On the other hand, our choice of $S$ is not unique:
there is a basis $s^{(1)}$, \dots, $s^{(g+2)}\in H^0(X,{-}K_X)$
consisting of eigensections. This basis gives us $G$-invariant divisors 
$S^{(1)}$, \dots, $S^{(g+2)}$ generating $|{-}K_X|$.
By the above $\Sing(X)\subset S^{(i)}$ for all $i$.
Thus $\Sing(X)\subset \cap S^{(i)}=\Bs |{-}K_X|$.
This contradicts Lemma \ref{lemma-base-point-free}.
\hfill$\Box$ \end{proof}
\def\proofname{Proof}


\def\cprime{$'$} \def\polhk#1{\setbox0=\hbox{#1}{\ooalign{\hidewidth
 \lower1.5ex\hbox{`}\hidewidth\crcr\unhbox0}}}

\end{document}